\documentclass[12pt]{amsart}

\usepackage{amscd,amssymb,amsthm,verbatim}

\newcommand{\diam}{\operatorname{diam}}

\newcommand{\Dir}{\operatorname{Dir}}

\newcommand{\dvol}{\operatorname{dvol}}

\newcommand{\Image}{\operatorname{Im}}

\newcommand{\R}{{\mathbb R}}

\newcommand{\supp}{\operatorname{supp}}

\numberwithin{equation}{section}
\theoremstyle{plain}

\newtheorem{lemma}[equation]{Lemma}
\newtheorem{theorem}[equation]{Theorem}
\newtheorem{proposition}[equation]{Proposition}

\theoremstyle{definition}

\theoremstyle{definition}
\newtheorem{example}[equation]{Example}

\theoremstyle{definition}
\newtheorem{remark}[equation]{Remark}
\def\<{\langle}
\def\>{\rangle}
\def\({\left(}
\def\){\right)}

\begin{document}

\title{On tangent cones in Wasserstein space}
\author{John Lott}
\address{Department of Mathematics\\
University of California - Berkeley\\
Berkeley, CA  94720-3840\\ 
USA} \email{lott@berkeley.edu}

\thanks{Research partially supported
by NSF grant DMS-1207654 and a Simons Fellowship}
\date{August 6, 2016}
\subjclass[2000]{}

\begin{abstract}
If $M$ is a smooth compact Riemannian manifold, let $P(M)$ 
denote the Wasserstein
space of probability measures on $M$. If $S$ is an embedded
submanifold of $M$, and $\mu$ is an absolutely continuous measure on $S$,
then we compute the tangent cone of $P(M)$ at $\mu$.
\end{abstract}

\maketitle

\section{Introduction} \label{section1}

In optimal transport theory, a displacement interpolation is a
one-parameter family of measures that represents the most efficient
way of displacing mass between two given probability measures.
Finding a displacement interpolation between two
probability measures
is the same as finding a
minimizing geodesic in the space of probability
measures, equipped with the Wasserstein metric $W_2$
\cite[Proposition 2.10]{Lott-Villani (2009)}.
For background on optimal transport and Wasserstein space, we refer
to Villani's book \cite{Villani (2009)}.

If $M$ is a compact connected Riemannian manifold with
nonnegative sectional curvature
then $P(M)$ is a compact length space with nonnegative curvature
in the sense of Alexandrov
\cite[Theorem A.8]{Lott-Villani (2009)},
\cite[Proposition 2.10]{Sturm (2006)}.
Hence one can define the tangent cone $T_{\mu} P(M)$
of $P(M)$ at a measure $\mu \in P(M)$.
If $\mu$ is absolutely continuous with respect to the volume form
$\dvol_M$ then $T_{\mu} P(M)$ is a Hilbert space
\cite[Proposition A.33]{Lott-Villani (2009)}. More generally, 
one can define tangent cones of $P(M)$ without any curvature assumption
on $M$,
using Ohta's $2$-uniform structure on $P(M)$ \cite{Ohta (2009)}.
Gigli showed that
$T_{\mu} P(M)$ is a Hilbert space if and only if $\mu$ is a ``regular''
measure, meaning that it gives zero measure to any hypersurface which,
locally, is the graph of the difference of two convex functions
\cite[Corollary 6.6]{Gigli (2011)}. 
It is natural to ask
what the tangent cones are at other measures. 

A wide class of tractable measures comes from submanifolds.
Suppose that $S$ is a smooth embedded submanifold of a
compact connected Riemannian manifold $M$. Suppose that
$\mu$ is an absolutely continuous probability measure on $S$. We can
also view $\mu$ as an element of $P(M)$. For simplicity, we assume that
$\supp(\mu) = S$. 

\begin{theorem} \label{theorem1.1}
We have
\begin{equation} \label{1.2}
T_{\mu} P(M) = H \oplus \int_{s \in S} P_2(N_sM) \: d\mu(s),
\end{equation}
where
\begin{itemize}
\item  $H$ is the Hilbert space of gradient vector fields
$\overline{\Image(\nabla)} \subset L^2(TS, d\mu)$,
\item $N_sM$ is the normal space to $S \subset M$ at $s \in S$ and
\item $P_2(N_sM)$ is the metric cone of probability measures on $N_sM$
with finite second moment, equipped with the $2$-Wasserstein metric.
\end{itemize}
\end{theorem}

The homotheties in the metric cone structure on $P_2(N_sM)$ arise from
radial rescalings of $N_sM$. The direct sum and integral in (\ref{1.2}) refer
to computing square distances.

The proof of Theorem \ref{theorem1.1} 
amounts to understanding optimal transport
starting from a measure supported on a submanifold. This seems to be a
natural question in its own right which has not been considered much.
Gangbo and McCann proved results about
optimal transport between measures supported on hypersurfaces in 
Euclidean space \cite{Gangbo-McCann (2000)}. 
McCann-Sosio and Kitagawa-Warren gave more refined results about optimal
transport between two measures supported on a sphere
\cite{Kitagawa-Warren (2012),McCann-Sosio (2011)}.
Castillon considered optimal transport between a measure supported on a
submanifold of Euclidean space and a measure supported on a linear subspace
\cite{Castillon (2010)}.

In the setting of Theorem \ref{theorem1.1}, 
a Wasserstein geodesic $\{\mu_t\}_{t \in [0,
\epsilon]}$ starting from $\mu$ consists of a family of geodesics shooting
off from $S$ in various directions. 
The geometric meaning of Theorem \ref{theorem1.1} 
is that the tangential component
of these directions is the gradient of a function on $S$. To motivate
this statement, in Section \ref{section2} we give a Benamou-Brenier-type
variational approach to the problem of optimally tranporting a
measure supported on one hypersurface to a measure supported on a disjoint
hypersurface, through a family of measures supported on hypersurfaces.
One finds that the only constraint is the aforementioned tangentiality
constraint. The rigorous proof of Theorem \ref{theorem1.1} is in 
Section \ref{section3}.

The structure of this paper is as follows.  In Section \ref{section2} we
give a formal derivation of the equation for optimal transport between
two measures supported on disjoint hypersurfaces of a Riemannian manifold.
The derivation is based on a variational method. In Section \ref{section3}
we prove Theorem \ref{theorem1.1}.

I thank C\'edric Villani for helpful comments, and
Robert McCann for references to the literature. I thank the
referee for his/her remarks.

\section{Variational approach} \label{section2}

Let $M$ be a smooth closed Riemannian manifold. Let
$S$ be a smooth closed manifold and let $S_0, S_1$ be
disjoint codimension-one submanifolds of $M$ diffeomorphic to $S$.
Let $\rho_0 \dvol_{S_0}$ and $\rho_1 \dvol_{S_1}$ be smooth
probability measures on $S_0$ and $S_1$, respectively.
We consider the problem of optimally transporting 
$\rho_0 \dvol_{S_0}$ to $\rho_1 \dvol_{S_1}$ through
a family of measures supported on codimension-one
submanifolds $\{S_t\}_{t \in [0,1]}$.
We will specify the intermediate submanifolds to be level sets of
a function $T$, which in turn will become one of the
variables in the optimization problem.

We assume that there is a
codimension-zero submanifold-with-boundary $U$ of $M$,
with $\partial U = S_0 \cup S_1$.
We also assume that there is a smooth submersion $T : U \rightarrow [0,1]$
so that $T^{-1}(0) = S_0$ and $T^{-1}(1) = S_1$.
For $t \in [0,1]$, put $S_t = T^{-1}(t)$. 
These are the intermediate hypersurfaces.

We now want to describe a family of measures
$\{\mu_t\}_{t \in [0,1]}$ that live
on the hypersurfaces
$\{S_t\}_{t \in [0,1]}$. It is convenient to think of these
measures as fitting together to form a measure on $U$.
Let $\mu$ be a
smooth measure on $U$. 
In terms of the fibering $T : U \rightarrow [0,1]$, decompose
$\mu$ as 
$\mu = \mu_t dt$ with $\mu_t$ a measure on $S_t$.
We assume that $\mu_0 = \rho_0 \dvol_{S_0}$ and 
$\mu_1 = \rho_1\dvol_{S_1}$.

Let $V$ be a vector field on $U$. 
We want the flow $\{\phi_s\}$ of $V$ to send level sets
of $T$ to level sets.
Imagining that there is an external clock, it's convenient
to think of $S_t$ as the evolving hypersurface at time $t$.
Correlating the flow of $V$ with the clock gives the constraint
\begin{equation} \label{2.1}
VT=1.
\end{equation}
Then
$\phi_s$ maps $S_t$ to $S_{t+s}$.

We also want the flow to be compatible with the measures
$\{\mu_t\}_{t \in [0,1]}$ in the sense that $\phi_s^* \mu_{t+s} = \mu_t$.
Now $\phi_s^* dT = d \phi_s^* T =  d(T+s) = dT$, so it is equivalent
to require that $\phi_s^*$ preserves the measure $\mu = 
\mu_t dt$. This gives the constraint
\begin{equation} \label{2.2}
{\mathcal L}_V \mu = 0.
\end{equation}
In particular, each $\mu_t$ is a probability measure.

To define a functional along the lines of Benamou and Brenier
\cite{Benamou-Brenier (2000)},
put 
\begin{equation} \label{2.3}
E = \frac12 \int_U |V|^2 \: d\mu =
\frac12 \int_0^1 \int_{S_t} |V|^2 \: d\mu_t \: dt. 
\end{equation}
We want to minimize $E$ under the constraints
${\mathcal L}_V \mu = 0$, $VT = 1$,
$\mu_0 = \rho_0 \dvol_{S_0}$ and
$\mu_1 = \rho_1 \dvol_{S_1}$.
Let
$\phi$ and $\eta$ be new functions on $U$, which
will be Lagrange multipliers for the constraints.
Then we want to extremize
\begin{equation} \label{2.4}
{\mathcal E} = \int_U \left[ \frac12 |V|^2 \: d\mu
+ \phi {\mathcal L}_V d\mu + \eta (VT-1) d\mu \right] 
\end{equation}
with respect to $V$, $\mu$, $\phi$ and $\eta$.

We will use the equations
\begin{align} \label{2.5}
\int_U \phi {\mathcal L}_V d\mu
= & \int_U \left[ 
{\mathcal L}_V(\phi d\mu)
- ({\mathcal L}_V \phi) d\mu
 \right] \\
= & - \int_U 
(V \phi) d\mu
+ \int_{S_1} \phi(1) d\mu_1 -
\int_{S_0} \phi(0) d\mu_0 \notag
\end{align}
and
\begin{align} \label{2.6}
\int_U \eta VT d\mu = &
\int_U \left[ {\mathcal L}_V (T \eta d\mu) - 
T {\mathcal L}_V (\eta d\mu) \right] \\
= &- \int_U T {\mathcal L}_V (\eta d\mu) 
+ \int_{S_1} \eta(1) d\mu_1.  \notag
\end{align}

The Euler-Lagrange equation for $V$ is
\begin{equation} \label{2.7}
V - \nabla \phi + \eta \nabla T = 0.
\end{equation}
The Euler-Lagrange equation for $\mu$ is
\begin{equation} \label{2.8}
\frac12 |V|^2 - V\phi = 0.
\end{equation}
Varying $T$ gives
\begin{equation} \label{2.9}
0 = {\mathcal L}_V(\eta d\mu) = (V\eta) d\mu,
\end{equation}
so the Euler-Lagrange equation for $T$ is
\begin{equation} \label{2.10}
V\eta = 0.
\end{equation}

Substituting (\ref{2.7}) into (\ref{2.8}) gives
$|\nabla \phi|^2 = \eta^2 |\nabla T|^2$, so
$\eta = \pm \frac{|\nabla \phi|}{|\nabla T|}$.
Then (\ref{2.7}) becomes
\begin{equation}
V = \nabla \phi \mp \frac{|\nabla \phi|}{|\nabla T|} \nabla T.
\end{equation}
Equation (\ref{2.1}) gives
\begin{equation} \label{added}
1 = \langle \nabla \phi, \nabla T \rangle \mp |\nabla \phi| \cdot
|\nabla T|.
\end{equation}
If the ``$\mp$'' is ``$-$'' then the right-hand side of
(\ref{added}) is nonpositive, which is a contradiction.
Thus 
\begin{equation}
1 = \langle \nabla \phi, \nabla T \rangle+ |\nabla \phi| \cdot
|\nabla T|
\end{equation}
and
\begin{equation}
V = \nabla \phi + \frac{|\nabla \phi|}{|\nabla T|} \nabla T.
\end{equation}
Equation (\ref{2.10}) becomes
\begin{equation}
V \frac{|\nabla \phi|}{|\nabla T|} = 0,
\end{equation}
which is equivalent to  
\begin{equation} \label{2.14}
\frac12 V |V|^2 = 0.
\end{equation}
Equation (\ref{2.14}) says that $V$ has constant
length along its flowlines.
The measure $\mu$ must still satisfy the conservation law (\ref{2.2}).

From (\ref{2.8}), the evolution of $\phi$ between level sets 
is given by
\begin{equation} \label{2.16}
V\phi = \frac12 |V|^2 = \frac12 \frac{|\nabla \phi|}{|\nabla T|}.
\end{equation}

The normal line to a level set $S_t$ is spanned by $\nabla T$.
It follows from (\ref{2.7}) that the tangential part of $V$
is the gradient of a function on $S_t$ :
\begin{equation} \label{2.17}
V_{tan} = \nabla_{S_t} \left( \phi \Big|_{S_t} \right).
\end{equation}
The normal part of $V$ is 
\begin{equation} \label{2.18}
V_{norm} = 
\frac{\langle V, \nabla T \rangle}{|\nabla T|^2} \nabla T =
\frac{1}{|\nabla T|^2} \nabla T,
\end{equation}
as must be the case from (\ref{2.1}).

The conclusion is that the tangential part of $V$ on $S_t$ is a gradient
vector field on $S_t$, while the normal part of $V$ on $S_t$ is unconstrained.

\section{Tangent cones} \label{section3}

\subsection{Optimal transport from submanifolds}

Let $M$ be a smooth closed Riemannian 
manifold.  Let $i : S \rightarrow M$ be an
embedding.

Let $\pi : TM \rightarrow M$ be the projection map.
Given $\epsilon > 0$, define
$E_\epsilon : TM \rightarrow TM$ by 
$E_\epsilon(m,v) = \left( \exp_m (\epsilon v),
d(\exp_m)_{\epsilon v} \epsilon v \right)$.  
We define $\pi^S$ and $E^S_\epsilon$ similarly, replacing $M$ by $S$.

Put $T_S M = i^* TM$, a vector bundle on $S$ with projection map
$\pi_{T_SM} : T_S M \rightarrow S$. There is an orthogonal
splitting $T_S M = TS \oplus N_S M$ into the tangential part and
the normal part. Let
$\pi_{N_SM} : N_S M \rightarrow S$
be the projection to the base of $N_SM$.
Given
$v \in TS$, let $v^T \in TS$ denote its tangential part and let
$v^\perp \in NS$ denote its normal part. Let
$p^T : T_SM \rightarrow TS$ be the orthogonal projection.

A function $F : S \rightarrow \R \cup \{\infty\}$ is {\it semiconvex}
if there is some $\lambda \in \R$ so that for all minimizing
constant-speed geodesics $\gamma : [0,1] \rightarrow S$, we have
\begin{equation} \label{3.1}
F(\gamma(t)) \le t F(\gamma(1)) + (1-t) F(\gamma(0)) - \frac12
\lambda t (1-t) d_S(\gamma(0), \gamma(1))^2
\end{equation}
for all $t \in [0,1]$.

Suppose that $F$ is a semiconvex function on $S$. Then $(s, w) \in TS$
lies in the {\em subdifferential set}
$\nabla^- F$ if for all $w^\prime \in T_sS$,
\begin{equation} \label{3.2}
F(s) + \langle w, w^\prime \rangle \le F(\exp_s w^\prime) +
o(|w^\prime|).
\end{equation}

Define the cost function $c : S \times M \rightarrow \R$ by $c(s,x) = \frac12
d(s,x)^2$. Given $\eta : M \rightarrow \R \cup \{ - \infty \}$, its
$c$-transform is the function
$\eta^c : S \rightarrow \R \cup \{\infty\}$ given by
\begin{equation} \label{3.3}
\eta^c(s) = \sup_{x \in M} \left( \eta(x) - \frac12 d^2(s,x) \right).
\end{equation}
Given $\psi : S \rightarrow \R \cup \{ \infty \}$, its
$c$-transform is the function 
$\psi^c : M \rightarrow \R \cup \{ - \infty\}$ given by
\begin{equation} \label{3.4}
\psi^c(x) = \inf_{s \in S} \left( \psi(s) + \frac12 d^2(s,x) \right).
\end{equation}
A function $\psi : S \rightarrow \R \cup \{ \infty \}$ is 
{\em $c$-convex}
if $\psi = \eta^c$ for some 
$\eta : M \rightarrow \R \cup \{ - \infty \}$.
A function $\eta : M \rightarrow \R \cup \{ -\infty \}$ is 
{\em $c$-concave}
if $\eta = \psi^c$ for some 
$\psi : S \rightarrow \R \cup \{ \infty \}$.

From \cite[Proposition 5.8]{Villani (2009)},
a function $F : S \rightarrow \R \cup \{- \infty\}$ is 
$c$-convex if and only if $F = (F^c)^c$, i.e. for all $s \in S$,
\begin{equation} \label{3.5}
F(s) = \sup_{x \in M} \inf_{s^\prime \in S}
\left( F(s^\prime) + \frac12 d^2(s^\prime,x) -
\frac12 d^2(s,x)  \right).
\end{equation}

The next lemma appears in \cite[Lemma 2.9]{Gigli (2011)} when $S=M$.

\begin{lemma} \label{lemma3.6}
If $F : S \rightarrow \R \cup \{\infty\}$ is a semiconvex function
then there is some $\epsilon > 0$ so that $\epsilon F$ is
$c$-convex.
\end{lemma}
\begin{proof}
Clearly 
\begin{equation} \label{3.7}
\epsilon F(s) \ge \sup_{x \in M} \inf_{s^\prime \in S}
\left( \epsilon F(s^\prime) + \frac12 d^2(s^\prime,x) -
\frac12 d^2(s,x)  \right),
\end{equation}
as is seen by taking $s^\prime = s$ on the right-hand side of (\ref{3.7}).
Hence we must show that for suitable $\epsilon > 0$, for all $s \in S$
we have
\begin{equation} \label{3.8}
\epsilon F(s) \le \sup_{x \in M} \inf_{s^\prime \in S}
\left( \epsilon F(s^\prime) + \frac12 d^2(s^\prime,x) -
\frac12 d^2(s,x)  \right).
\end{equation}
For this, it suffices to show that for each $s \in S$, there is
some $x \in M$ so that
\begin{equation} \label{3.9}
\epsilon F(s) \le 
\inf_{s^\prime \in S}
\left( \epsilon F(s^\prime) + \frac12 d^2(s^\prime,x) -
\frac12 d^2(s,x)  \right).
\end{equation}
That is, it suffices to show that for each $s \in S$, there is
some $x \in M$ so that for all $s^\prime \in S$, we have 
\begin{equation} \label{3.10}
\epsilon F(s) \le 
\epsilon F(s^\prime) + \frac12 d^2(s^\prime,x) -
\frac12 d^2(s,x),
\end{equation}
i.e. 
\begin{equation} \label{3.11}
\epsilon F(s) + \frac12 d^2(s,x) \le 
\epsilon F(s^\prime) + \frac12 d^2(s^\prime,x).
\end{equation}

We know that $F$ is $K$-Lipschitz for some $K < \infty$
\cite[Theorem 10.8 and Proposition 10.12]{Villani (2009)}.
Hence if $v \in \nabla^-_s F$ then $|v| \le K$.
Given $s$, choose $v \in \nabla^-_s F$ and put $x = \exp_s (
\epsilon v) \in M$. Then $d(s,x) \le \epsilon K$.

Put $G(s^\prime) =  
\epsilon F(s^\prime) + \frac12 d^2( s^\prime,x)$.
We want to show that $G(s) \le G(s^\prime)$ for all $s^\prime \in S$.
Suppose not.
Let $s^\prime$ be a minimum point for $G$; then
$G(s^\prime) < G(s)$.

We claim first that $s^\prime \in B_{4 \epsilon K}(s)$.
To see this, if $d(s,s^\prime) \ge 4 \epsilon K$ then since
\begin{equation} \label{3.12}
d(s^\prime,x) \ge d(s,s^\prime) - d(s,x) \ge
d(s,s^\prime) - \epsilon K,
\end{equation}
we have
\begin{align} \label{3.13}
\frac12 d^2(s^\prime,x) - \frac12 d^2(s,x) & \ge
\frac12 \left( d(s,s^\prime) -  \epsilon K \right)^2 - 
\frac12 \left( \epsilon K \right)^2 \\
& = \frac12 
(d(s,s^\prime) - 2 \epsilon K) \cdot d(s,s^\prime) \notag \\
& \ge \epsilon K d(s,s^\prime) \ge \epsilon(F(s) - F(s^\prime)), \notag
\end{align}
which contradicts that $G(s^\prime) < G(s)$. This proves the claim.

If $10 \epsilon K$ is less than the injectivity radius of $M$ then there is a
unique minimizing geodesic from $s$ to $x$, and its tangent
vector at $s$ is $\epsilon v$.
It follows that $0 \in \nabla_s^- G$.
Finally, since $d(s,x) \le 
\epsilon K$, we can choose an $\epsilon$ (depending on $K$, $S$ and $M$)
to ensure that $G$ is strictly 
convex on $B_{4\epsilon K}(s)$, with the latter being a totally convex set.
Considering the function $G$ along a minimizing geodesic from
$s$ to $s^\prime$, we obtain a contradiction to the assumed strict
convexity of $G$, along with the facts that $0 \in \nabla_s^- G$ and
$0 \in \nabla_{s^\prime}^- G$. 

Thus $G$ is minimized at $s$, which
implies (\ref{3.11}).
\end{proof}

Let $\nu$ be a compactly-supported probability measure on $T_S M
\subset TM$.
Let $L < \infty$ be such that the support of $\nu$ is contained in
$\{v \in T_SM \: : \: |v| \le L \}$.
Put $\mu_\epsilon = \pi_* (E_\epsilon)_* \nu$.

\begin{proposition} \label{proposition3.14}
a. Let $f$ be a semiconvex function on $S$.  Suppose that
$\nu$ is supported on $\{v \in T_SM \: : \:  v^T \in \nabla^- f \}$.
Then there is some $\epsilon > 0$ so that the $1$-parameter family of 
measures $\{\mu_t\}_{t \in [0, \epsilon]}$ is
a Wasserstein geodesic. \\
b. Given $\nu$, suppose that
for some $\epsilon > 0$, the $1$-parameter family of 
measures $\{\mu_t\}_{t \in [0, \epsilon]}$ is
a Wasserstein geodesic.  Then there is a 
semiconvex function $f$ on $S$ so that
$\nu$ is supported on $\{v \in T_SM \: : \:  v^T \in \nabla^- f \}$.\\
\end{proposition}
\begin{proof}
a. For $t > 0$, define $\eta_t : M \rightarrow \R$ by
$\eta_t = (tf)^c$. From Lemma \ref{lemma3.6}, if $t$ is small enough then
$tf$ is $c$-convex. It follows from 
\cite[Proposition 5.8]{Villani (2009)} that
$(\eta_t)^c = tf$.

From \cite[Theorem 5.10]{Villani (2009)}, 
if a set $\Gamma_t \subset S \times M$
is such that $\eta_t(x) = tf(s) + \frac12 d^2(s,x)$ for all
$(s,x) \in S \times M$ then any
probability measure $\Pi_t$ with support in $\Gamma_t$ is an optimal
transport plan. 
We take
\begin{equation} \label{3.15}
\Gamma_t = \{(s,x) \in S \times M \: : \: \eta_t(x) = tf(s) + \frac12 d^2(s,x) \}.
\end{equation}
Now $\eta_t(x) = tf(s) + \frac12 d^2(s,x)$ if
for all $s^\prime \in S$, we have
\begin{equation} \label{3.16}
t f(s) + \frac12
d^2(s,x) \le  t f(s^\prime) + \frac12
d^2(s^\prime,x).
\end{equation}

To prove part a. of the proposition, 
it suffices to show that for all sufficiently
small $t$, equation (\ref{3.16}) is satisfied for 
$s,s^\prime \in S$ and $x = \exp_s (tv)$, where
$v \in T_sM$ lies in the support of $\nu$ and 
satisfies $v^T \in  \nabla^- f$.

Given $s$ and $v$, we know that $d(s,x) \le tL$. Put 
$G(s^\prime) = t f(s^\prime) +
\frac12 d^2(s^\prime, x)$. Let $s^\prime$ be a minimum point of
$G$ and suppose, to get a contradiction, that $G(s^\prime) < G(s)$.

Let $K < \infty$ be the Lipschitz constant of $f$.
We claim first that 
$s^\prime \in B_{t(2K+2L)}(s)$.
To see this, if $d(s,s^\prime) \ge t(2K+2L)$ then
\begin{equation} \label{3.17}
d(s^\prime,x) \ge d(s,s^\prime) - d(s,x) \ge
d(s,s^\prime) - tL
\end{equation}
and
\begin{align} \label{3.18}
\frac12 d^2(s^\prime,x) - \frac12 d^2(s,x) & \ge
\frac12 \left( d(s,s^\prime) - tL \right)^2 - (tL)^2 \\
& = \frac12 (d(s,s^\prime) - 2tL) \cdot d(s,s^\prime) \notag \\
& \ge tK d(s,s^\prime) \ge t(f(s) - f(s^\prime)), \notag
\end{align}
which is a contradiction and proves the claim.

There is some $\epsilon > 0$ (depending on $L$, $S$ and $M$)
so that if $t \in [0, \epsilon]$ then we are ensured that
there is a unique minimizing geodesic from $s$ to $x$, and its
tangent vector at $s$ is $tv$. It follows that $0 \in \nabla_s^- G$.
Finally, since $d(s,x) \le \epsilon L$, we can choose $\epsilon$
(depending on $K$, $L$, $S$ and $M$) to ensure that $G$ is strictly
convex on $B_{t(2K+2L)}(s)$, the latter being totally convex. 
Considering the function $G$ along a
minimizing geodesic from $s$ to $s^\prime$, we obtain a contradiction
to the assumed strict convexity of $G$, along with the facts that
$0 \in \nabla_s^- G$ and $0 \in \nabla_{s^\prime}^- G$.
This proves part (a) of the proposition.

Now suppose that $\{\mu_t\}_{t \in [0,\epsilon]}$ is a Wasserstein
geodesic.  From \cite[Theorem 5.10]{Villani (2009)}, there is a $c$-convex
function $\epsilon f$ on $S$ so that if we define 
its conjugate $(\epsilon f)^c$ using (\ref{3.4}) then
$\{(s, \exp_s(\epsilon v)\}_{(s,v) \in \supp(\nu)}$
is contained in 
\begin{equation} \label{3.19}
\Gamma_\epsilon = 
\left\{(s,x) \in S \times M \: : \: (\epsilon f)^c(x) =
\epsilon f(s) + \frac12 d^2(s,x) \right\}.
\end{equation}
That is, for all $s^\prime \in S$,
\begin{equation} \label{3.20}
\epsilon f(s) + \frac12 d^2(s,\exp_s(\epsilon v)) 
\le \epsilon f(s^\prime) + 
\frac12 d^2(s^\prime,\exp_s(\epsilon v)).
\end{equation}

Without loss of generality, we can shrink $\epsilon$ as desired.
Define a curve in $S$ by $s^\prime(u) = \exp_s(- u w^\prime)$ where 
$w^\prime \in T_sS$, $u$ varies over
a small interval $(-\delta, \delta)$ and $\exp_s$ denotes here the
exponential map for the submanifold $S$.
Let $\{\gamma_u : [0, \epsilon] \rightarrow M \}_{u \in 
(- \delta, \delta)}$ be a smooth
$1$-parameter family with $\gamma_0(t) = \exp_s(tv)$,
$\gamma_u(0) = s^\prime(u)$ and $\gamma_u(\epsilon) =
\exp_s(\epsilon v)$. Let $L(u)$ be the length of $\gamma_u$.
Then 
\begin{equation} \label{3.21}
\epsilon f(s) + \frac12 d^2(s,\exp_s(\epsilon v)) 
\le \epsilon f(s^\prime(u)) + 
\frac12 L^2(u).
\end{equation}
By the first variation formula, 
\begin{equation} \label{3.22}
\frac{d}{du} \Big|_{u=0} \frac12 L^2(u) = \epsilon \langle v^T, 
w^\prime \rangle.
\end{equation}
It follows that $\epsilon v^T \in \nabla_s^- (\epsilon f)$, so
$v^T \in \nabla_s^- f$.
\end{proof}

\begin{remark}
The phenomenon of possible nonuniqueness, 
in the normal component of the optimal transport between two measures supported
on convex hypersurfaces in Euclidean space, was recognized in
\cite[Proposition 4.3]{Gangbo-McCann (2000)}.
\end{remark}

\begin{example} \label{example3.23} Put $M = S^1 \times \R$. (It is noncompact, but this will
be irrelevant for the example.) Let $F \in C^\infty(S^1)$ be a positive
function. Put $S = \{(x, F(x)) : x \in S^1\}$. 
Define $p : S \rightarrow S^1 \times \{0\}$ by
$p(x, F(x)) = (x,0)$.
Let $\mu_0$ be a
smooth measure on $S$. Put $\mu_1 = p_* \mu_0$. The Wasserstein
geodesic from $\mu_0$ to $\mu_1$ moves the measure down along
vertical lines.  Defining $f$ on $S$ by
$f(x, F(x)) = - \frac12 \left( F(x) \right)^2$, one finds that
$v^T = \nabla f$. Compare with \cite[Corollary 2.6]{Castillon (2010)}.
\end{example}

\subsection{Tangent cones}

If $X$ is a complete length space with Alexandrov curvature bounded below
then one can define the tangent cone $T_xX$ at $x \in X$ as follows.
Let $\Sigma_x^\prime$ be the space of equivalence classes of minimal geodesic
segments emanating from $x$, with the equivalence relation identifying
two segments if they form a zero angle at $x$ (which means that one
segment is contained in the other). The metric on $\Sigma_x^\prime$ is
the angle.  By definition, the space of directions $\Sigma_x$ is the
metric completion of $\Sigma_x^\prime$. The tangent cone $T_xX$ is the
union of $\R^+ \times \Sigma_x$ and a ``vertex'' point, with the metric
described in \cite[\S 10.9]{Burago-Burago-Ivanov (2001)}.

If $X$ is finite-dimensional then one can also describe $T_x X$ as the
pointed Gromov-Hausdorff limit $\lim_{\lambda \rightarrow \infty}
\left( \lambda X, x \right)$. This latter description doesn't make sense
if $X$ is infinite-dimensional, whereas the preceding definition does.

If $M$ is a smooth compact connected Riemannian manifold, and it has
nonnegative sectional curvature, then $P(M)$ has nonnegative Alexandrov
curvature and one can talk about a tangent cone $T_{\mu} P(M)$
\cite[Appendix A]{Lott-Villani (2009)}. If
$M$ does not have nonnegative sectional curvature then $P(M)$ will not
have Alexandrov curvature bounded below.  Nevertheless, one can still
define $T_{\mu} P(M)$ in the same way \cite[Section 3]{Ohta (2009)}.

As a point of terminology, what is called a tangent cone here, and in
\cite{Lott-Villani (2009)},
is called the ``abstract tangent space'' in
\cite{Gigli (2011)}. The linear part of the tangent cone is called the
``tangent space'' in \cite{Ambrosio-Gigli (2008)} and the
``space of gradients'' or ``tangent vector fields'' in \cite{Gigli (2011)}.

A minimal geodesic segment emanating from $\mu \in P(M)$ is determined by
a probability measure $\Pi$ on the space of constant-speed minimizing geodesics
\begin{equation}
\Gamma = \{ \gamma : [0,1] \rightarrow M \: : \: L(\gamma) =
d_M(\gamma(0), \gamma(1)) \},
\end{equation}
which has the property that under the time-zero 
evaluation $e_0 : \Gamma \rightarrow M$,
we have $(e_0)_* \Gamma = \mu$ 
\cite[Section 2]{Lott-Villani (2009)}. The corresponding geodesic segment
is given by $\mu_t = (e_t)_* \Pi$, where $e_t : \Gamma \rightarrow M$ 
is time-$t$ evaluation.

Using this characterization of minimizing geodesic segments, one can
describe $T_{\mu} P(M)$ as follows. With $\pi : TM \rightarrow M$ being
projection to the base, put 
\begin{equation}
P_2(TM)_{\mu} = \{ \nu \in P_2(TM) \: : \: \pi_* \nu = \mu \},
\end{equation}
where $P_2$ refers to measures with finite second moment.
Given $\nu^1, \nu^2 \in P_2(TM)_\mu$, decompose them as
\begin{equation}
\nu^i = \int_M \nu^i_m \: d\mu(m), 
\end{equation} 
with $\nu^i_m \in P_2(T_mM)$. Define $W_{\mu} (\nu^1, \nu^2)$ by
\begin{equation}
W_{\mu}^2 (\nu^1, \nu^2) = \int_M W_2^2(\nu^1_m, \nu^2_m) \: d\mu(m).
\end{equation}
Let $\Dir_\mu$ be the set of elements $\nu \in P_2(TM)_\mu$ with the
property that $\{\pi_* (E_t)_* \nu \}_{t \in [0, \epsilon]}$ describes
a minimizing Wasserstein geodesic for some $\epsilon$.  Then
$T_{\mu} P(M)$ is isometric to the metric completion of $\Dir_{\mu}$ 
with respect to $W_\mu$ \cite[Theorem 5.5]{Gigli (2011)}.

We note that since $M$ is compact, any element of $\Dir_\mu$ has compact
support.  This is because for $\nu$-almost all $v \in TM$, the
geodesic $\{\exp_{\pi(v)} tv \}_{t \in [0, \epsilon]}$ must be
minimizing \cite[Proposition 2.10]{Lott-Villani (2009)}, so
$|v| \le \epsilon^{-1} \diam(M)$. \\ \\
{\it Proof of Theorem \ref{theorem1.1} : } From Proposition 
\ref{proposition3.14}, 
$\Dir_{\mu}$ is the set of compactly-supported measures 
$\nu \in P(T_SM) \subset P(TM)$ so that $\pi_* \nu = \mu$ and
there is a semiconvex function $f$ on $S$ such that
$\nu$ has support on $\{v \in T_SM \: : \: v^T \in \nabla^- f\}$.
Because $\mu$ has full support on $S$ by assumption, $\nabla^- f$ is
single-valued at $\mu$-almost all $s \in S$. Equivalently, there is a
compactly-supported $\nu^N \in P(N_SM)$, which decomposes under
$\pi_{N_SM} \: : \: N_SM \rightarrow S$ as
$\nu^N = \int_S \nu^N_s \: d\mu(s)$ with
$\nu^N_s \in P_2(N_sM)$, so that for all 
$F \in C(T_SM) = C(TS \oplus N_SM)$,
we have
\begin{equation}
\int_{T_SM} F \: d\nu = \int_S \int_{N_sM} F(\nabla^-f(s), w) \:
d\nu^N_s(w) \: d\mu(s). 
\end{equation}
Given two such measures $\nu^1, \nu^2$, it follows that
\begin{equation} \label{w2}
W^2_{\mu}(\nu^1, \nu^2) = \int_S \langle \nabla^-f^1, \nabla^-f^2 \rangle
\: d\mu + \int_S W_2^2(\nu^{1,N}_s, \nu^{2,N}_s) \: d\mu(s).
\end{equation}
Upon taking the metric completion of $\Dir_{\mu}$, the tangential term 
in (\ref{w2}) gives
the closure of the space of gradient vector fields
in the Hilbert space $L^2(TS, d\mu)$ of
square-integrable sections of $TS$ 
\cite[Proposition A.33]{Lott-Villani (2009)}.
The normal term gives
$\int_{s \in S} P_2(N_sM) \: d\mu(s)$, where the metric comes from
the last term in (\ref{w2}). This proves the theorem. \qed 

\begin{remark}
In Section 2 we considered transports in which the intermediate
measures were supported on hypersurfaces. This corresponds to
Wasserstein geodesics starting from $\mu$ for which the initial
velocity, as an element of $T_\mu P(M)$, comes from a section of
$T_SM$. In terms of Theorem \ref{theorem1.1}, this means that the data
for the initial velocity consisted of a gradient vector field 
$\nabla \phi$ on $S$
and a section ${\mathcal N}$ of $N_SM$, with the element of
$P_2(N_sM)$ being the delta measure at ${\mathcal N}(s)$.
\end{remark}

\subsection{Gauss map as an optimal transport map}

In this subsection, which is an addendum to the preceding subsections,
we give an example of optimal transport coming from the Gauss map of 
a convex hypersurface in $\R^n$.

Let $S$ be the boundary of a compact convex subset of $\R^n$. We assume
that near any point, $S$ is locally the graph of a $C^2$-regular function.  Let
$N : S \rightarrow S^{n-1}$ be the outward unit normal.  
Let $\kappa \in C^0(S)$ be the Gaussian curvature function, the
product of the principal values. Then
$N_* (\kappa \dvol_S) = \dvol_{S^{n-1}}$.

The optimal transport plans in $\R^n$ for the cost function
$\frac12 |m_1-m_2|^2$ are the same as those for
the cost function $- \langle m_1, m_2 \rangle$.
Given $R > 0$, $s \in S$ and $x \in S^{n-1}$, the cost 
function of the points $s$ and $Rx$ becomes $- \: R \langle s,x \rangle$.
Considering an optimal transport problem between $S$ and
$R \cdot S^{n-1}$,
the optimal transport plans for the cost function 
$- \: R \langle s,x \rangle$ are the same as those for the cost function
$- \: \langle s,x \rangle$.
This motivates considering the cost function 
$c : S \times S^{n-1} \rightarrow \R$ 
given by $c(s,x) \: = \: - \langle s, x \rangle$.
Here we imagine taking $R \rightarrow \infty$ so that 
$S^{n-1}$ is a ``sphere at infinity'', not
an embedded sphere in $\R^n$, although when we write $\langle s, x \rangle$
we are treating $x$ as a unit vector.

The analog of (\ref{3.15}) is
\begin{equation} 
\Gamma_t = \{(s,x) \in S \times S^{n-1} \: : \: 
\eta_t(x) = tf(s) - \langle s,x \rangle \}.
\end{equation}
Now $\eta_t(x) = tf(s) - \langle s,x \rangle$ if
for all $s^\prime \in S$, we have
\begin{equation} 
t f(s) - \langle s,x \rangle
\le  t f(s^\prime) - \langle s^\prime, x \rangle.
\end{equation}

Taking $f = 0$, one sees that for all $s \in S$ we have
$(s, N(s)) \in \Gamma_1$, since the convexity of $S$ implies that
$\langle s^\prime - s, N(s) \rangle \le 0$ for all $s^\prime \in S$.
Hence $N$ is an optimal transport map from the measure $\kappa \dvol_S$ 
on $S$, to the measure $\dvol_{S^{n-1}}$ on $S^{n-1}$.

\begin{remark}
In a different direction, 
Aleksandrov's problem of realizing a given curvature function
was related to optimal transport on a sphere in \cite{Oliker (2007)}, 
using a
certain cost function; see also
\cite{Bertrand (2015)}.
\end{remark}

\end{document}